\documentclass[11pt]{article}
\usepackage{amsmath, amssymb, amsthm, verbatim,enumerate,bbm}
\usepackage{indentfirst}
\usepackage{MnSymbol}
\usepackage{tikz}
\usepackage{mwe}
\usepackage[hyperfootnotes=false]{hyperref}

\date{\today}
\allowdisplaybreaks

\theoremstyle{plain}
\newtheorem{theorem}{Theorem}[section]
\newtheorem{lemma}[theorem]{Lemma}

\newtheorem{corollary}[theorem]{Corollary}
\newtheorem{conjecture}[theorem]{Conjecture}
\newtheorem{question}[theorem]{Question}
\theoremstyle{definition}

\newtheorem{remark}[theorem]{Remark}

\newcommand{\N}{\mathbb N } 
\newcommand{\Z}{\mathbb Z }

\DeclareMathOperator{\ord}{ord}

\title{Power maps in finite groups}

\author{ Matt Larson
\thanks{Department of Mathematics, Yale University. Email:
matthew.larson@yale.edu.} }

\begin{document}
\maketitle

\renewcommand{\thefootnote}{}



\renewcommand{\thefootnote}{\arabic{footnote}}
\setcounter{footnote}{0}

\begin{abstract}
\noindent In recent work, Pomerance and Shparlinski have obtained results on the number of cycles in the functional graph of the map $x \mapsto x^a$ in $\mathbb{F}_p^*$. We prove similar results for other families of finite groups.  In particular, we obtain estimates for the number of cycles for cyclic groups, symmetric groups, dihedral groups and $SL_2(\mathbb{F}_q)$. We also show that the cyclic group of order $n$ minimizes the number of cycles among all nilpotent groups of order $n$ for a fixed exponent $a$. Finally, we pose several problems.
\end{abstract}

\section{Introduction}

\noindent Let $H$ be a finite group, and let $a \ge 2$ be an integer. The iterations of the map $x \mapsto x^a$ form a sort of dynamical system in a finite group. As such, it is natural to study the structure of the periodic points of this map. Define the undirected multigraph $G(a, H)$ with vertex set $H$ and $x \sim y$ if $x^a = y$, with an additional edge if $y^a = x$. Note that $G(a, H)$ may have loops (for example at the identity) or cycles of length $2$. The orbit structure of the map $x \mapsto x^a$ in $G$ is encoded in $G(a, H)$. This graph has been extensively studied in the case of $H = (\Z/n\Z)^*$ in connection with algorithmic number theory and cryptography (see, e.g., \cite{Chou2004}, \cite{Kurlberg2005} and \cite{Vasiga2004}). In particular, the properties of the well-known Blum-Blum-Shub psuedorandom number generator \cite{Blum1986} are determined by the properties of $G(2, (\Z/pq\Z)^*)$.
\begin{figure}[!htb]
\begin{center}
\begin{tikzpicture}
\draw (1,1) -- (4,1) -- (4,4) -- (1,4) -- (1,1);
\draw (7, 1) -- (7, 4);
\draw (-0.7, -0.7) -- (1,1);
\draw (4, 1) -- (5.7, -0.7);
\draw (1, 4) -- (-0.7, 5.7);
\draw (4, 4) -- (5.7, 5.7);
\node at (7.3,  4) {0}; 
\node at (7.3,  1) {5};
\node at (1.4,  1.4) {2}; 
\node at (1.4,  4.4) {4}; 
\node at (4.4,  1.4) {6};
\node at (3.6,  4.4) {8};  
\node at (-1, -0.7) {1}; 
\node at (6, -0.7) {3};
\node at (6, 5.7) {9};
\node at (-1, 5.7) {7};
\fill (7, 4) circle[radius= 0.1 cm];
\fill (7, 1) circle[radius= 0.1 cm];
\fill (1, 1) circle[radius= 0.1 cm];
\fill (4, 4) circle[radius= 0.1 cm];
\fill (4, 1) circle[radius= 0.1 cm];
\fill (1, 4) circle[radius= 0.1 cm];
\fill (5.7, -0.7) circle[radius= 0.1 cm];
\fill (-0.7, 5.7) circle[radius= 0.1 cm];
\fill (5.7, 5.7) circle[radius= 0.1 cm];
\fill (-0.7, -0.7) circle[radius= 0.1 cm];
%\draw (7,4) .. controls (8,5) .. (7, 6)
% .. (6, 5) .. (6,4);
\draw[scale=4] (1.75, 1)  to[in=50,out=130,loop] (1.75, 1);
\end{tikzpicture}
\caption{$G(2, \mathbb{Z}/10\mathbb{Z})$}
\end{center}
\end{figure}
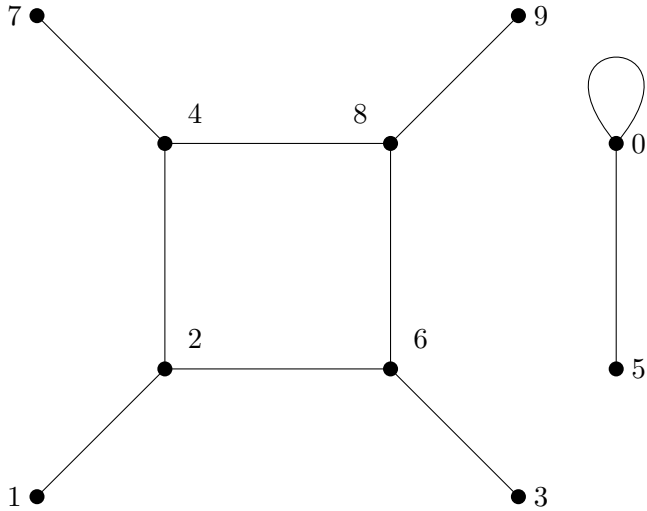

\noindent Note that $G(a, H)$ is a refinement of the power graph of $H$ (see \cite{Abawajy2013} and references therein). In particular, the power graph of $H$  is the graph with vertex set $H$ and $x \sim y$ if $x \in \langle y \rangle$ or $y \in \langle x \rangle$. One can build the power graph of $H$ out of  $G(a,H)$ by taking the union of the edges of $G(a, H)$ for $1 \le a \le \vert H \vert$ and deleting any loops or multiple edges.

\noindent Let $N(a, H)$ denote the number of connected components in $G(a, H)$. Since each connected component contains a unique cycle, $N(a, H)$ is also the number of cycles in $G(a, H)$.
In recent work, Pomerance and Shparlinski gave results on the average order, normal order, and extremal  order of $N(a, \mathbb{F}_p^*)$ for $p$ prime.

\begin{theorem}[{\cite[Theorems 1.1 and 1.2]{Pomerance2017}}]\label{pom}
For any $a \ge 2$:
\begin{itemize}
\item There exist infinitely many primes $p$ such that $N(a, \mathbb{F}_p^*) > p^{5/12 + o(1)}$. 
\item For almost all primes $p$, $N(a, \mathbb{F}_p^*) < p^{1/2 + o(1)}$.
 \item $\frac{1}{\pi(x)} \displaystyle\sum_{p \le x} N(a, \mathbb{F}_p^*) \gg x^{0.293}$ 
\end{itemize} Under the assumption of the Elliot-Halberstam conjecture and a strong Linnik's constant, we can improve this to $$\frac{1}{\pi(x)} \sum_{p \le x} N(a, \mathbb{F}_p^*) \ge x^{1+ o(1)}.$$

\end{theorem}

\noindent Pomerance and Shparlinski asked for an extension of these results to other groups. We consider the question of the size of $N(a, G)$ for various families of groups. Using results from number theory, group theory, and probability theory, we obtain results on the size of $N(a, G)$ for cyclic groups, dihedral groups, symmetric groups and the special linear group of degree $2$ over a finite field. 

\noindent Next, we conjecture that, for any $a$, the cyclic groups have the fewest connected components over any groups of a given order. More precisely, 

\begin{conjecture} \label{minconjecture}
\noindent Let $G$ be a group of order $n$. Then $$N(a, G) \ge N(a, C_n).$$
\end{conjecture}

\noindent We have verified this conjecture using Sage \cite{Stein2017} for all groups of order at most 1000, except for groups of order $768$, if $a\in\{2, 3, \dotsc, 20\}$. We prove the following partial result: 
\begin{theorem} \label{nilpotent}
Let $G$ be a nilpotent group of order $n$. Then $$N(a, G) \ge N(a, C_n).$$
\end{theorem}

\noindent In Section $2$, we introduce results used to estimate $N(a, G)$. In Section $3$, we estimate the normal order, average order, and extremal order of $N(a, G)$ for several families of groups. In Section $4$, we prove theorem \ref{nilpotent}. In Section $5$, we discuss further directions and ask several questions.

\subsection{Notation}
\noindent Throughout this paper, $p$ denotes a prime number, $q$ denotes a prime power, and $a$ denotes a positive integer at least $2$. All groups are finite, and group multiplication is always written multiplicatively. 

\noindent For a set $A$, we denote the characteristic function of $A$ by $1_A(x)$. For $g \in G$ a group, let $\vert g \vert$ denote the order of $g$. Let $\ord_{n}(a)$ denote the multiplicative order of $a$ in $\mathbb{Z}/n\mathbb{Z}$. For a group $G$, let $w_G(d)$ denote the number of elements of order $d$. We will often write $w(d)$ for $w_G(d)$ if the group is obvious. Let $C_n$ denote the cyclic group of order $n$, $D_{n}$ denote the dihedral group of order $2n$, $SL_n(\mathbb{F}_q)$ denote the special linear group of degree $n$ over the finite field of $q$ elements and let $S_n$ denote the symmetric group of order $n!$. Let $\lambda$ denote the Carmichael lambda function, i.e. $\lambda(n)$ is the exponent of $(\Z/n\Z)^*$. Let $\varphi$ denote the Euler $\varphi$-function.

\noindent We use standard Vinogradov notation and Landau notation. Recall that the statements $U = O(V)$, $U \ll V$ and $V \gg U$ all mean $\vert U \vert \le cV$ for some $c >0$. We also use the notation $o(1)$ to denote a quantity that tends to $0$ as some parameter goes to infinity.  The dependency of the constant on a parameter will be denoted as a subscript. We say almost all elements of a set $S \subseteq \mathbb{N}$ have a property $P$ if the proportion of the elements of $S$ that have $P$ and are at most $n$ is $1 + o(1)$.

\section{General tools} 
\noindent Our main tool for estimating $N(a, G)$ is the following lemma:
\begin{lemma}\label{formula}Let $\rho$ denote the largest factor of $\vert G \vert$ relatively prime to $a$. Then \begin{equation*}
N(a, G) = \sum_{d \vert \rho} \frac{w(d)}{\ord_{d}(a)}.
\end{equation*}
\end{lemma}

\begin{proof}
We generalize an argument of Chou and Shparlinski in \cite{Chou2004}. Consider the map $x \mapsto x^a$. Then let $t \ge 0, c > 0$ minimal such that $ x^{a^t} =  x^{a^{t + c}}$ for all $x$, which exist since the map $x \mapsto x^a$ is preperiodic.  Let  $d$ denote the order of $x$. Then $d \vert a^t (a^c - 1)$, so $t = 0$ if and only if $\gcd(a, d) = 1$. If $t = 0$, then $x$ lies in a cycle of length $\ord_d(a)$, and there are $w(d)$ elements that lie in such cycles, showing the result. 
\end{proof}

\noindent We will often use this result in the form $$N(a, G) = \sum_{\substack{g \in G \\ \gcd(\vert g \vert, a) = 1}} \frac{1}{\ord_{\vert g \vert}(a)},$$ which follows from lemma \ref{formula} by grouping terms by order.
We observe that if a group $G$ has many element of large order, then $N(a, G)$ is likely to be small. This gives some justification to conjecture \ref{minconjecture}.
We will also make use of the following lemma.
\begin{lemma}\label{sum} Let $H_1, \dotsc, H_n \le G$, and suppose $H_i \cap H_j = \{e\}$ for $i \not = j$, where $e$ is the identity of $G$. Then
$$N(a, G) \ge \sum_{i=1}^{n} N(a, H_i) - n + 1.$$
\end{lemma}

\begin{proof}
Note that the subgraph in $G(a, G)$ induced by $H_i$ is isomorphic to $G(a, H_i)$. These induced subgraphs overlap only at the identity, and, in these induced subgraphs, each connected component contains a unique cycle. In $G(a, G)$, these induced subgraphs cannot be connected to each other, except for the connected component containing the identity. 
\end{proof}

\noindent Before proving the last general result, we state a lemma

\begin{lemma}\label{ord} If $\frac{d d'}{\gcd(d, d')} = n$, then $\ord_d(a) \ord_{d'}(a) \ge \ord_n(a)$.
\end{lemma}
\begin{proof}
As $d \mid a^{\ord_d(a)} - 1$ and $d' \mid a^{\ord_{d'}(a)} - 1$, $n \mid a^{\ord_d(a) \ord_{d'}(a)} - 1$.
\end{proof}

\begin{theorem} Let $G, H$ finite groups. Then 
$$N(a, G \times H) \ge N(a, G) N(a, H).$$
\end{theorem}
\begin{proof}
Let $\rho_1$ and $\rho_2$ be the largest divisors of $\vert G \vert$ and $\vert H \vert$ coprime to $a$ respectively. Then
\begin{equation*} \begin{split}
\left ( \sum_{d \mid \rho_1} \frac{w_G(d)}{\ord_{d}(a)} \right) \left (\sum_{d' \mid \rho_2} \frac{w_H(d')}{\ord_{d'}(a)} \right) & = \sum_{d \mid \rho_1, d' \mid \rho_2} \frac{w_G(d) w_H(d')}{\ord_d(a) \ord_{d'}(a)} \\
& = \sum_{k \mid \rho_1 \rho_2} \sum_{\substack{ d \mid \rho_1, d' \mid \rho_2,\\  d d'/\gcd(d, d') = k}} \frac{w_G(d) w_H(d')}{\ord_d(a) \ord_{d'}(a)} \\
& \le \sum_{k \mid \rho_1 \rho_2} \sum_{\substack{ d \mid \rho_1, d' \mid \rho_2,\\  d d'/\gcd(d, d') = k}} \frac{w_G(d) w_H(d')}{\ord_k(a)} \\
& = \sum_{k \mid \rho_1 \rho_2} \frac{w_{G \times H}(k)}{\ord_k(a)} \\
& = N(a, G \times H),
\end{split} \end{equation*}
where in the inequality we use lemma \ref{ord}.
\end{proof}

\section{Size of $N(a, G)$}

\subsection{Cyclic groups}

\noindent We show results on the average order, normal order, and extremal order of $N(a, C_n)$.

\begin{theorem}\label{normc_n} Let $\delta = 0.2961$. Then \begin{equation*}
\frac{1}{x} \sum_{n \le x} N(a, C_n) \ge x^{1 - \delta + o(1)}.
\end{equation*}
\end{theorem}

\begin{theorem} \label{cyclicextremal}
For any fixed $a$, there exist infinitely many $n$ such that $$N(a, C_n) \ge n^{1 + o(1)}.$$
\end{theorem}

\begin{theorem}\label{cyclicnormal} For almost all $n$, we have that 
$$N(a, C_n) \le n^{1/2 + o(1)}.$$
\end{theorem}

\begin{remark}
Under the Elliott-Halberstam conjecture, we can remove $\delta$ from theorem \ref{normc_n}, i.e. we can show that $\frac{1}{x} \sum_{n \le x} N(a, C_n) \ge x^{1 + o(1)}$. Under the generalized Riemann hypothesis, we can remove the $1/2$ from theorem \ref{cyclicnormal} and show that for almost all $n$, $N(a, C_n) \le n^{o(1)}$.
\end{remark}

\noindent In conjunction with the following lemma, the above theorems immediately give results on dihedral groups.

\begin{lemma}
If $a$ is even, then $N(a, D_n) = N(a, C_n)$. If $a$ is odd, then $N(a, D_n) = n + N(a, C_n)$.
\end{lemma}

\begin{proof}
Recall that $D_n$ consists of a cyclic subgroup of order $n$ and $n$ elements of order $2$ lying outside this cyclic subgroup. If $a$ is even, then each element of order $2$ is connected to the component that contains the identity. If $a$ is odd, then each element of order $2$ lies in a component that consists of a single vertex with a loop.
\end{proof}

\begin{proof}[Proof of theorem \ref{normc_n}]
We use the strategy of Pomerance and Shparlinski in \cite{Pomerance2017}. First we recall a result of Baker and Harman.

\begin{lemma}[{\cite{Baker1998}, Theorem 1}] \label{shifted_smooth}
There is an absolute constant $\kappa$ with the following property: Let $x$ sufficiently large, and let $$v = \frac{\log x}{\log \log x}, \quad w = v^{1/0.2961}.$$ Let $$\mathcal{Q} = \{p \in \left [\frac{w}{(\log w)^{\kappa}}, w \right] \enskip : \enskip p - 1 \mid M_v\},$$ where $M_v$ is the least common multiple of the integers in $[1, v]$. Then $$\vert \mathcal{Q} \vert \ge \frac{w}{(\log w)^\kappa}.$$
\end{lemma}

\noindent Now we prove the result. Let $\mathcal{Q}$ be the set of primes given by lemma \ref{shifted_smooth}. Let $$k = \left \lfloor \frac{\log x}{\log w} \right \rfloor.$$ Let $\mathcal{S}$ denote the set of products of $k$ distinct elements of $\mathcal{Q}$. We see that $$\vert \mathcal{S} \vert = \binom{\vert \mathcal{Q} \vert}{k} = \left( \frac{w}{k} \right)^k x^{o(1)},$$ using that $(n/k)^k \le \binom{n}{k} \le (ne/k)^k$.  We compute that $k^k = x^{0.2961 + o(1)}$ and $w^k = x^{1 + o(1)}$, so
$$\vert \mathcal{S} \vert = x^{1 - 0.2961 + o(1)}.$$
We also note that, for any $m \in \mathcal{S}$, $$x \ge w^{k} \ge m \ge (w/(\log w)^{\kappa})^k = x^{1 + o(1)}.$$
By lemma \ref{ord}, we have that for any $m \in \mathcal{S}$, $\ord_m(a) \mid M_v$. By the prime number theorem, this implies that $$\ord_m(a) \le M_v = \exp(v(1 + o(1))) = x^{o(1)}.$$
Therefore, for each $m \in \mathcal{S}$, we have $$N(a, C_m) \ge \frac{\varphi(m)}{\ord_m(a)} = x^{1 + o(1)},$$ since $\varphi(m) = m^{1 + o(1)} = x^{1 + o(1)}$ \cite[Theorem 328]{Hardy}. Therefore we have found $x^{1 - 0.2961 + o(1)}$ positive integers $m$ less than or equal to $x$ such that $N(a, C_m) = x^{1 + o(1)}$,  which implies the result.

\begin{remark}
One can obtain the same result by using the work of Ambrose; it follows from the specialization to $\mathbb{Q}$ of { \cite[Theorem 1]{Ambrose2014}}. As we can remove $\delta$ from the result of Ambrose under the Elliott-Halberstam conjecture, we can show that the average value of $N(a, C_n)$ is $x^{1 + o(1)}$ under the Elliott-Halberstam conjecture. 
\end{remark}
\end{proof}

\begin{proof}[Proof of theorem \ref{cyclicextremal}]
Let $k$ be a large integer, and set $n = a^k - 1$. Then, using \cite[Theorem 328]{Hardy},
\begin{equation*} N(a, C_n) \ge \frac{\varphi(a^k -1)}{k} \gg \frac{n}{\log n \log \log n}.
\end{equation*}
\end{proof}

\noindent Before proving theorem \ref{cyclicnormal}, we recall some properties of the Carmichael lambda function.

\begin{lemma}[{\cite[Lemma 2]{Friedlander2000}}] \label{lambdabound}
If $d \vert n$, then $$\varphi(d)/\lambda(d) \vert \varphi(n)/\lambda(n).$$
\end{lemma}

\begin{theorem}[{\cite[Theorem 2]{Erdos1991}}]\label{lambda} For almost all $n$, 
$$\lambda(n) = n^{1 + o(1)}.$$
\end{theorem}

\begin{lemma}[{\cite[Lemma 5]{Kurlberg2005}}]\label{ord_bound} We have
$$\ord_{n}(a) \ge \frac{\lambda(n)}{n} \prod_{p \vert n} \ord_p(a).$$
\end{lemma}

\noindent Let $B$ denote the set of primes such that $\ord_p(a) < \sqrt{p}/\log p$.

\begin{lemma}[\cite{Erdos1999}]\label{density} With $B$ defined as above,   $\vert B \cap \{1, \dotsc, N \} \vert = O(n/ (\log n)^3)$.
\end{lemma}

\begin{remark}
Using the results in \cite{Kurlberg2003}, we can show that the set of primes $p$ at most $n$  for with $\ord_p(a) \le p^{1 + o(1)}$ has size $O(n/(\log n)^3)$ under the generalized Riemann hypothesis, which would lead to a corresponding improvement in theorem \ref{cyclicnormal} to $N(a, C_n) \le n^{o(1)}$ for almost all $n$. 
\end{remark}
\noindent For an integer $n$, let $n_B$ denote the largest divisor of $n$ that is a product of primes from $B$. 
\begin{lemma} \label{bad}
For almost all $n \le N$, $n_B < \log n$.
\end{lemma}
 
\begin{proof}
By the density estimate in lemma \ref{density}, we see that $$\sum_{n = n_B} \frac{1}{n} = \prod_{p \in B} \left(1 - \frac{1}{p} \right)^{-1} = O(1).$$
Therefore, for any $\varepsilon > 0$, there is $C = C(\varepsilon)$ such that $$\sum_{\substack{n = n_B, \\ n> C}} \frac{1}{n} < \varepsilon.$$ Thus for all but $\varepsilon N$ integers $n \le N$, we have that $n_B < C$. As $\varepsilon$ was arbitrary and eventually $\log n > C$, this proves the claim. 
\end{proof}

\begin{lemma}[{\cite[Lemma 7]{Kurlberg2005}}]\label{core}
The number of positive integers $n$ at most $x$ such that there is a positive integer $s$ such that $s^2 \mid n$ and $s^2 \ge \log n$ is $O\left( \frac{x}{\log x} \right)$.
\end{lemma}

\begin{proof}[Proof of Theorem \ref{cyclicnormal}]
By lemma \ref{lambda}, lemma \ref{bad}, and lemma \ref{core} there is a set $S$ of density $1$ such that $n_B < \log n$, $s^2 < \log n$ for every $s$ such that $s^2$ divides $n$, and  $\lambda(n) = n^{1 + o(1)}$ for all $n \in S$. By lemma \ref{ord_bound}, we have that $$N(a, C_n) \le \sum_{d \vert n} \frac{d \varphi(d)}{\lambda(d) \prod_{p \vert d} \ord_p(a)}.$$

\noindent Using the bound that $\varphi(n) < n$ and lemma \ref{lambdabound}, in form of $\varphi(d)/\lambda(d) \le \varphi(n)/\lambda(n)$ for $d \vert n$, we have that, for almost all $n$, 
\begin{equation*} \begin{split}
N(a, C_n) & \le \sum_{d \vert n} \frac{d \varphi(d)}{\lambda(d) \prod_{p \vert d} \ord_p(a)} \\
& \le \sum_{d \vert n} \frac{d \varphi(n)}{\lambda(n) \prod_{p \vert d} \ord_p(a)} \\
& = \sum_{d \vert n} \frac{d n^{o(1)}}{\prod_{p \vert d} \ord_p(a)} \\
& \le n^{1/2 + o(1)},
\end{split} \end{equation*}
where in the last inequality we are using that the square part of $n$ is at most $\log n$ and the product of the primes in $B$ dividing $n$ is at most $\log n$. \end{proof}

\subsection{Symmetric groups}

\noindent As lemma \ref{sum} implies that the sequence $\{N(a, S_n) \}_{n \in \N}$ is non-decreasing, since $S_{n-1}$ embeds into $S_n$, it makes less sense to discuss the average order, normal order, and extremal order of $N(a, S_n)$. We therefore prove bounds on the size of $N(a, S_n)$.

\begin{theorem} \label{S_n} We have
$$N(a, S_n) \ge \frac{ n!}{\exp \left ( \frac{\varphi(a)}{2a} \log^2 n (1 + o(1)) \right)}.$$
\end{theorem}

\noindent Let $T_n = T_n(a)$ denote the set of permutations in $S_n$ with order coprime to $a$, and let $S(n)$ denote the set of positive integers coprime to $a$ that are at most $n$. We will use concentration bounds to show that almost all elements of $T_n$ have large order, and then we use the trivial bound that $\ord_d(a) \le d$ to bound $N(a, S_n)$. 

\begin{theorem}[{\cite[Theorem 1]{Pavlov1998}}] There exists constants $C = C(a)$ and $\delta = \delta(a)$ such that 
$$\vert T_n \vert = C(n-1)! n^{\varphi(a)/a} + O((n-1)! n^{\varphi(a)/a - \delta}).$$
\end{theorem}

\begin{lemma}[{\cite[Theorem 1]{Yakymiv2018}}]\label{normal} For some permutation $\sigma$, let $M(\sigma)$ denote the order of the permutation. Choose a random permutation $\tau_n$ from $T_n$. Then \begin{equation*}
P \left ( \frac{ \log M(\tau_n) - \sum_{i \in S(n)} (\log i)/i}{\sqrt{\sum_{i \in S(n)} (\log i)^2/i} } \le x \right) \xrightarrow{d} \Phi(x),
\end{equation*}
where $\Phi(x)$ is the standard normal distribution and $\xrightarrow{d}$ denotes convergence in distribution. 
\end{lemma}
\noindent We use lemma \ref{normal} to bound the order of most elements of $T_n$ and then use the trivial upper bound on $\ord_d(a)$. First we obtain an asymptotic for $\sum_{i \in S(n)} (\log i)^2/i$. 

\begin{lemma} We have
$$\sum_{i \in S(n)} \frac{\log i}{i} = \frac{\varphi(a)}{2a} \log^2 n + o(\log^2 n).$$
\end{lemma}
\begin{proof}
Observe that, using partial summation,
\begin{equation*} \begin{split}
\sum_{i \in S(n)} \frac{\log i}{i} & = \log n \sum_{i=1}^{n} \frac{1_{S(n)}(i)}{i} + \sum_{m=1}^{n-1} (\log m - \log(m + 1)) \sum_{i=1}^m \frac{1_{S(n)}(i)}{i}.
\end{split} \end{equation*}
and \begin{equation*} \begin{split}
\sum_{i=1}^{n} \frac{1_{S(n)}(i)}{i} &= \sum_{m=1}^{n-1} \left ( \frac{1}{m} - \frac{1}{m+1} \right) m \frac{\varphi(a)}{a} + O(1) \\
& = \frac{\varphi(a)}{a} \log n + O(1).
\end{split} \end{equation*}
Using partial summation again, we have that $$\sum_{i=1}^n \frac{\log i}{i} = \log^2 n + \sum_{m=1}^{n-1} (\log m - \log m + 1)\log m + O(\log n).$$
On the other hand, $$\sum_{i=1}^n \frac{\log i}{i} = \int_{1}^{n} \frac{\log x}{x} dx + o(\log n) = \frac{\log^2 n}{2} + o(\log n).$$
Hence $$\sum_{m=1}^{n-1} (\log m - \log m + 1)\log m = \frac{\log^2 n}{2} + o(\log n),$$ showing that $$\sum_{i \in S(n)} \frac{\log i}{i} = \frac{\varphi(a)}{2a} \log^2 n + o(\log^2 n).$$
\end{proof}

\begin{proof}[Proof of theorem \ref{S_n}]
We have the trivial bound $$\sum_{i \in S(n)} \frac{(\log i)^2}{i} = O(\log^3 n).$$
For all but $o_a(\vert T_n \vert)$ permutations $\tau_n$ in $T_n$, we have that $$\log M(\tau_n) \le \sum_{i \in S(n)} \frac{\log i}{i} + O(\log \log n (\log n)^{3/2}).$$
Hence, for almost all permutations in $T_n$, we have that \begin{equation*} \label{orderbound}M(\tau_n) \le \exp \left ( \frac{\varphi(a)}{2a} \log^2 n (1 + o(1)) \right). \end{equation*}
In order to turn this into a lower bound for $N(a, S_n)$, we need a lower bound on $\ord_d(a)$ for $d$ coprime to $a$. Using the trivial bound that $\ord_d(a) \le d \le \exp \left ( \frac{\varphi(a)}{2a} \log^2 n (1 + o(1)) \right)$ for almost all permutations $\tau_n \in T_n$, we have that $$N(a, S_n) \gg_a \frac{ (n-1)! n^{\varphi(a)/a}}{\exp \left ( \frac{\varphi(a)}{2a} \log^2 n (1 + o(1)) \right)} = \frac{ n! }{\exp \left ( \frac{\varphi(a)}{2a} \log^2 n (1 + o(1)) \right)}.$$ \end{proof}

\noindent We conjecture that this lower bound is of the correct order, as the trivial bound $\ord_d(a) \le d$ is usually fairly sharp for most $d$. Without finer control over the orders of permutations than is known, it seems difficult to prove a sharp upper bound. However, we can show that

$$N(a, S_n) = o_a((n-1)! n^{\varphi(a)/a}).$$
Indeed, by lemma \ref{normal}, we have that for all but $o_a(\vert T_n \vert)$ elements of $T_n$, 
$$M(\tau_n) \ge \exp \left ( \frac{\varphi(a)}{2a} \log^2 n (1 + o(1)) \right).$$
Hence $$N(a, S_n) \le o_a(\vert T_n \vert) + \frac{ (n-1)! n^{\varphi(a)/a}}{\exp \left ( \frac{\varphi(a)}{2a} \log^2 n (1 + o(1)) \right)} = o_a((n-1)! n^{\varphi(a)/a}).$$

\subsection{Special linear groups over finite fields}
\noindent Because of highly explicit knowledge of the conjugacy class structure of $SL_2(\mathbb{F}_q)$, we are able to compute $N(a, SL_2(\mathbb{F}_q))$.
\begin{theorem}\label{SL(2,q)} Let $q = p^{c}$ be an odd prime power. If $\gcd(a, q) = 1$, then
$$N(a, SL_2(\mathbb{F}_q)) = \frac{q^2 - q}{2} N(a, C_{q+1}) + \frac{q^2 + q}{2} N(a, C_{q-1}) + (q^2 - 1)(1 + 1_{2 \nmid a})(\frac{1}{\ord_p(a)} - 1).$$
where $1_{2 \nmid a}$ is $1$ if $a$ is odd and $0$ otherwise. If $\gcd(a, q) > 1$, then the last term does not appear. 
\end{theorem}

\noindent Before we begin the proof, we recall some facts about conjugacy classes in $SL_2(\mathbb{F}_q)$ for $q$ odd. We break the conjugacy classes into $4$ types (see e.g. \cite{FH}):
\begin{itemize}
\item Type $1$: The $(q-3)/2$  conjugacy classes of elements which are diagonalizable of $\mathbb{F}_q$; they are parametrized by matrices of the form $\begin{pmatrix} \alpha & 0 \\ 0 & \alpha^{-1} \end{pmatrix}$ for $\alpha \in \mathbb{F}_q^* \setminus \{1, -1 \}$. Each conjugacy class has $q(q+1)$ elements.
\item Type $2$: The $(q-1)/2$ conjugacy classes of elements which are diagonalizable of $\mathbb{F}_{q^2}$ but not $\mathbb{F}_q$; they are parametrized by matrices of the form $\begin{pmatrix} \alpha & 0 \\ 0 & \alpha^{-1} \end{pmatrix}$ for $\alpha \in \mathbb{F}_{q^2}^* \setminus \{1, -1 \}$ and satisfying $\alpha \cdot Fr(\alpha) = \alpha^{q+1} = 1$, where $Fr$ denotes the Frobenius endomorphism. Each conjugacy class has $q(q-1)$ elements.
\item Type $3$: The central conjugacy classes $\{I\}$ and $\{-I\}$.
\item Type $4$: The $4$ conjugacy classes that are not semi-simple, which are parametrized by $\begin{pmatrix} 1 & 1\\ 0 & 1 \end{pmatrix},$ $\begin{pmatrix} -1 & 1\\ 0 & -1 \end{pmatrix},$ $\begin{pmatrix} 1 & b\\ 0 & 1 \end{pmatrix},$ $\begin{pmatrix} -1 & b\\ 0 & -1 \end{pmatrix}$, where $b$ is a non-square in $\mathbb{F}_q$. There are $(q^2 -1)/2$ elements in each conjugacy class.
\end{itemize}

\begin{proof} Note that the order of an element in a type $1$ conjugacy classes is just the order of the eigenvalue, and the eigenvalues, and the eigenvalues like in the cyclic group $\mathbb{F}_q^{*}$. Therefore, because each type $1$ conjugacy class has eigenvalues $\alpha$ and $\alpha^{-1}$, there are $\varphi(d)/2$ type $1$ conjugacy classes of order $d$ for each divisor $d$ of $q - 1$, $d \not= 1, 2$. Hence 
$$\sum_{a \in \mathcal{C},\text{ } \mathcal{C} \text{ type }1} \frac{1}{\ord_d(a)} = \frac{q^2 + q}{2} (N(a, C_{q - 1}) - 1 - 1_{2 \nmid a}).$$

\noindent Similarly for type $2$, we note that the elements of $\mathbb{F}_{q^2}$ satisfying $x^{q + 1} = 1$ form a cyclic subgroup of the multiplicative group. Therefore
$$\sum_{a \in \mathcal{C},\text{ } \mathcal{C} \text{ type }2} \frac{1}{\ord_d(a)} = \frac{q^2 - q}{2} (N(a, C_{q - 1}) - 1 - 1_{2 \nmid a}).$$

\noindent For type $3$, the contribution is $1 + 1_{2 \nmid a}$.

\noindent For type $4$, each element with eigenvalues $1$ has order $p$, and each element with eigenvalues $-1$ has order $2p$. Hence,
$$\sum_{a \in \mathcal{C},\text{ } \mathcal{C} \text{ type }4} \frac{1}{\ord_d(a)} = \frac{q^2 - 1}{ \ord_p(a)} + 1_{2 \nmid a} \frac{q^2 - 1}{\ord_{2p}(a)} = \frac{q^2 - 1}{\ord_p(a)}(1 + 1_{2\nmid a}),$$ since $\ord_{2p}(a) = \ord_p(a)$ for $a$ odd.
Summing over the $4$ types of conjugacy classes gives the result. 
\end{proof}

\noindent Then theorem \ref{pom} allows us to bound the normal and extremal order of $N(a, SL_2(\mathbb{F}_p))$, using that $N(a, SL_2(\mathbb{F}_p)) \gg p^2 N(a, \mathbb{F}_p^*)$.

\begin{corollary} \label{SL2Fp} 
There exists infinitely many primes $p$ such that $$N(a, SL_2(\mathbb{F}_p)) \ge p^{29/12 + o(1)}.$$
We also have 
$$\frac{1}{\pi(x)}\sum_{p \le x} N(a, SL_2(\mathbb{F}_p)) \gg x^{2.293}.$$
\end{corollary}

\section{On the minimal size of $N(a, G)$ among groups of a fixed order}
\noindent We now prove theorem \ref{nilpotent}. 
Our strategy is to show that the sum $\sum_{{g \in G, \\ \gcd(\vert g \vert, a) = 1}} \frac{1}{\ord_{\vert g \vert}(a)}$ majorizes $\sum_{{g \in C_{\vert G \vert}, \\ \gcd(\vert g \vert, a) = 1}} \frac{1}{\ord_{\vert g \vert}(a)}$ for any nilpotent group $G$. Then, lemma \ref{formula} immediately implies theorem \ref{nilpotent}.
Before proving theorem \ref{nilpotent}, we prove a lemma.
For a group $G$, let $B_G(n)$ denote the number of elements of order at least $n$ in $G$.
\begin{lemma}\label{p_ineq} Let $G$ be a group of order $p^k$. Then for all $n$,
$B_G(n) \le B_{C_{p^k}}(n)$.
\end{lemma}

\begin{proof}
First observe that the number of elements of order $n$ in any finite group is a multiple of $\varphi(n)$. Suppose $G$ is a counterexample to the lemma, then choose $\ell$ such that $B_G(p^\ell) > B_{C_{p^{k}}}(p^{\ell})$. Since $B_G(1) = B_{C_{p^k}}(1)$, there must be fewer than $\varphi(p^b)$ elements of order $p^b$ for some $b < \ell$. Hence there are no elements of order $p^b$ for some $b$. But if a group has an element of order $p^c$, then it also has an element of order $p^b$ for every $b < c$.
\end{proof}

\begin{proof}[Proof of theorem \ref{nilpotent}]
We first prove theorem \ref{nilpotent} for $p$-groups. Note that $\ord_{p^b}(a) \le \ord_{p^c}(a)$ if $b \le c$. But then lemma \ref{p_ineq} and lemma \ref{formula} immediately imply that $N(a, G) \ge N(a, C_{\vert G \vert})$ for any $p$-group $G$. 

\noindent Recall that a group is nilpotent if and only if it is a direct product of $p$-groups. Let $G = P_1 \times \dotsb \times P_k$ be a nilpotent group, and let $P_1, \dotsc, P_k$ be $p$-groups with orders $p_i^{e_i}$ for distinct primes $p_1, \dotsc, p_k$. Let $n = \vert G \vert$. We may assume that $\gcd(a, n) = 1$, as otherwise we may eliminate the $p$-groups with order not coprime to $a$. We need to show that $$\sum_{d \mid n} \frac{w(d)}{\ord_d(a)} \ge \sum_{d \mid n} \frac{\varphi(d)}{\ord_d(a)}. $$ 
Observe that, for a nilpotent group $G$, $w_G$ is a multiplicative function, i.e. $$w_G(p_1^{j_1}p_2^{j_2} \dotsb p_k^{j_k}) = w_{G}(p_1^{j_1}) w_{G}(p_2^{j_2}) \dotsb w_{G}(p_k^{j_k}).$$ 

\noindent We claim that for any set of $\ell$ primes, $p_{i_1}, \dotsc, p_{i_\ell}$, we have that $$\sum_{\substack{d \mid n \\ d = p_{i_1}^{b_1} \dotsb p_{i_\ell}^{b_\ell}}} \frac{w(d)}{\ord_d(a)} \ge \sum_{\substack{d \mid n \\ d = p_{i_1}^{b_1} \dotsb p_{i_\ell}^{b_\ell}}} \frac{\varphi(d)}{\ord_d(a)}.$$

\noindent This would clearly imply the result by summing over all subsets of the primes dividing $n$. We prove the claim by induction on $\ell$. The base case is the case of $p$-groups. Fix $b_1, \dotsc, b_{\ell-1}$ such that $p_{i_1}^{b_1} \dotsb p_{i_{\ell-1}}^{b_\ell-1} \mid n$. Then \begin{equation*} \begin{split}
\sum_{k=0}^{j_{i_\ell}} \frac{w(p_{i_1}^{b_1} \dotsb p_{i_{\ell-1}}^{b_\ell-1} p_{i_\ell}^k)}{\ord_{p_{i_1}^{b_1} \dotsb p_{i_{\ell-1}}^{b_\ell-1} p_{i_\ell}^k}(a)} & = w(p_{i_1}^{b_1} \dotsb p_{i_{\ell-1}}^{b_\ell-1}) \sum_{k=0}^{j_{i_{\ell}}} \frac{w(p_{i_{\ell}}^k)}{\ord_{p_{i_1}^{b_1} \dotsb p_{i_{\ell-1}}^{b_\ell-1} p_{i_\ell}^k}(a)} \\
& \ge \frac{w(p_{i_1}^{b_1} \dotsb p_{i_{\ell-1}}^{b_\ell-1})}{\ord_{p_{i_1}^{b_1} \dotsb p_{i_{\ell-1}}^{b_\ell-1}}(a)} \sum_{k=0}^{j_{i_{\ell}}} \frac{w(p_{i_{\ell}}^k)}{\ord_{p_{i_{\ell}}^k}(a)},
\end{split}\end{equation*} 
where we use lemma \ref{ord} in the inequality. The result follow from summing  over all choices of $b_1, \dotsc, b_{\ell-1}$ and the inductive hypothesis. 
\end{proof}

\section{Discussion}

\noindent In addition to proving a better upper bound on $N(a, S_n)$ and proving conjecture \ref{minconjecture}, we pose several open problems. 

\noindent Since the map $x \mapsto x^a$ is eventually periodic, the orbit $x, x^a, x^{a^2}, \dotsc$ consists of a tail which does not repeat followed by a cycle. If $x$ has no tail, then we say that $x$ is \textit{purely periodic}. Thus in $G(a, H)$, every purely periodic element has a rooted tree of tails leading into it. In \cite[Theorem 1]{Chou2004}, Chou and Shparlinski showed that if $H$ is cyclic, then all of the tails coming off the purely periodic elements in $H$ are isomorphic. In particular, every purely periodic element has tails of the same size. This enabled Chou and Shparlinski to give a simple expression for the average length of the period over all elements of $C_n$. Let $C(a, G)$ denote the average period of an element in $G$. Then 
\begin{theorem}[{\cite[Theorem 1]{Chou2004}}] If $\rho$ is the largest divisor of $n$ coprime to $a$, then
$$ C(a, C_n) = \frac{1}{\rho} \sum_{d \mid \rho} \varphi(d) \ord_d(a).$$
\end{theorem}

\noindent For general groups, the tails coming off a purely periodic vertex are not the same size. It would be interesting to compute or bound $C(a, G)$ for various families of groups.  

\noindent By analogy with the power graph, it would be interesting to determine what set of invariants is determined by $G(a, H)$ for some fixed $a$ or for all $a$. Groups $H$ of prime exponent and the same order clearly have the same $G(a, H)$ for every $a$. Using the example of Cameron and Ghosh in \cite{Cameron2011}, we see that, if $H = \langle x, y\quad  \vert \quad x^3 = y^3 = [x,y]^3= 1 \rangle$, the smallest non-abelian group of exponent $3$, then  $G(a, C_3 \times C_3 \times C_3) \cong G(a, H)$ for every $a$. This raises the following question:
\begin{question}
Are there groups $H$ and $K$ such that the power graph of $H$ is isomorphic to the power graph of $K$, but $G(a, H)$ is not isomorphic to $G(a, K)$ for some $a$?
\end{question}
\noindent It would be interesting to compute the asymptotics of $N(a, SL_n(\mathbb{F}_q))$ as $n$ grows, in analogy with the symmetric group. As in the case of $N(a, S_n)$, lemma \ref{sum} implies that the sequence $\{N(a, SL_n(\mathbb{F}_q))\}_{n \in \mathbb{N}}$ is non-decreasing since $SL_{n-1}(\mathbb{F}_q)$ embeds into $SL_{n}(\mathbb{F}_{q})$. 

\noindent One could also allow $a$ to vary. Let $\text{exp(G)}$ denote the exponent of $G$. Then clearly $N(a, G) = N(a + \text{exp(G)}, G)$. Then the following question is natural:
\begin{question}
What $a \in \{2, 3, \dotsc, \exp(G) - 1\}$ maximizes $N(a, S_n)$?
\end{question}

\section{Acknowledgments}
\noindent This research was conducted at the University of Minnesota Duluth REU and was supported by NSF DMS grant 1659047. We thank Igor Shparlinski for pointing out an oversight in an earlier version of the paper, Joe Gallian for suggesting the subject of the paper,  Arsen Yakymiv for discussing his work with us, and Joe Gallian and Phil Matchett Wood for reading the manuscript. We are also grateful to the referee for their helpful comments. 

\bibliography{power_map}
\bibliographystyle{plain}

\end{document}